\documentclass[12pt]{amsart}
\title{Generic Density\textbf{} of Periodic Orbits of Area-preserving Maps on Punctured Surfaces}
\author{Shaoyang Zhou}
\date{}

\usepackage{amsmath,amssymb}
\usepackage{amsthm}
\usepackage{mathrsfs}
\usepackage{graphicx}
\usepackage{latexsym}
\usepackage{verbatim}
\usepackage{enumitem}
\usepackage{dsfont}
\usepackage{amsbsy}
\usepackage{graphics}
\usepackage{xcolor}
\usepackage{tikz-cd}
\usepackage[bookmarksnumbered, bookmarksopen, colorlinks, citecolor= blue, linkcolor=black, urlcolor = black]{hyperref}
\usepackage[hmargin=2cm, vmargin=2.5cm]{geometry}

\usepackage{mathtools}

\numberwithin{equation}{section}

\newtheorem{thm}{Theorem}[section]
\newtheorem{prop}[thm]{Proposition}
\newtheorem{cor}[thm]{Corollary}
\newtheorem{lem}[thm]{Lemma}
\newtheorem{claim}[thm]{Claim}

\newtheorem{quest}[thm]{Question}

\theoremstyle{definition}
\newtheorem{rem}[thm]{Remark}

\newcommand{\End}{\mathrm{End}}

\newcommand{\cH}{\mathcal{H}}

\def\G{{\mathcal G}}
\def\N{{\mathbb N}}
\def\Z{{\mathbb Z}}
\def\cZ{{\mathcal Z}}
\def\Q{{\mathbb Q}}
\def\R{{\mathbb R}}

\def\cC{{\mathcal C}}
\def\P{{\mathcal P}}

\def\O{{\mathcal O}}

\def\H{{\mathcal H}}
\def\M{{\mathcal M}}
\def\X{{\mathcal X}}

\def\lv{{\lvert}}
\def\rv{{\rvert}}

\newcommand{\Diff}{\textrm{Diff}}
\newcommand{\Homeo}{\textrm{Homeo}}
\newcommand{\Ham}{\textrm{Ham}}
\newcommand{\Tw}{\textrm{TwPFH}}

\def\Sp{{\text{Sp}}}
\newcommand{\defeq}{\vcentcolon=}
\newcommand\restr[2]{{% we make the whole thing an ordinary symbol
  \left.\kern-\nulldelimiterspace % automatically resize the bar with \right
  #1 % the function
  \littletaller % pretend it's a little taller at normal size
  \right|_{#2} % this is the delimiter
  }}
\newcommand{\littletaller}{\mathchoice{\vphantom{\big|}}{}{}{}}

\begin{document}

\setcounter{tocdepth}{2}

\begin{abstract} In this paper, we study the dynamics of area-preserving maps in a non-compact setting. We show that the \(C^{\infty}\)-closing lemma holds for area-preserving diffeomorphisms on a closed surface with finitely many points removed. As a corollary, a \( C^{\infty} \)-generic area-preserving diffeomorphism on such a surface has a dense set of periodic points. For area-preserving maps on a finitely punctured 2-sphere, we establish a more quantitative result regarding the equidistribution of periodic orbits. The proof of this result involves a PFH Weyl law for rational area-preserving homeomorphisms, which may be of independent interest.
\end{abstract}

\maketitle

\begin{section}{Introduction}
\subsection{Setting}

Let \(\Sigma\) be a smooth closed surface equipped with an area form \(\omega\). Assume that \(\int_{\Sigma}\, \omega = 1\). Fix a non-empty, finite collection of points \(P = \{p_1, \dots, p_k\}\) in \(\Sigma\), and define \(\Sigma_P = \Sigma \setminus P\) and \(\omega_0 = \omega|_{\Sigma_P}\). Let \(\Diff(\Sigma_P, \omega_0)\) be the group of area-preserving diffeomorphisms of \((\Sigma_P, \omega_0)\). We consider two topologies on \(\Diff(\Sigma_P, \omega_0)\): the strong \(C^{\infty}\)-topology and the \(C_{\text{loc}}^{\infty}\)-topology. The details of these topologies are discussed in \cite{Hirsch}. Roughly speaking, the \(C_{\text{loc}}^{\infty}\)-topology controls the closeness of two maps and their derivatives over compact sets, while the strong \(C^{\infty}\)-topology has control all the way up to infinity (i.e., the set of punctures). 

\subsection{Statements of results}
We establish the following results concerning the periodic orbits of area-preserving maps of \((\Sigma_P, \omega_0)\):

\begin{thm}[$C^{\infty}$-closing lemma]
\label{main thm1}
Fix $\phi \in \Diff(\Sigma_{P},\omega_0)$ and an open set $U \subset \Sigma_{P}$. For any open neighborhood $V$ of $\phi$ in the strong $C^{\infty}$-topology, there exists $\phi' \in V$ such that $\phi'$ has a periodic orbit in $U$. 
\end{thm}

As \(\Diff(\Sigma_P, \omega_0)\) is a Baire space under the strong \(C^{\infty}\)-topology, applying the argument of \cite[Corollary 1.2]{Asaoka-Irie} yields the following corollary.

\begin{cor}[Generic density theorem]
A generic element of $\Diff(\Sigma_P,\omega_0)$ (in the strong $C^{\infty}$-topology) has a dense set of periodic points. 
\end{cor}

When \(\Sigma = S^2\), we also prove a result regarding the equidistribution of periodic orbits. The relevant terminologies are defined in Section 2.

\begin{thm}
\label{thm2}
A $C_{\text{loc}}^{\infty}$-dense element of $\Diff(S^2_P,\omega_0)$ has an equidistributed sequence of orbit sets. 
\end{thm}

Historically, Pugh and Robinson \cite{Pugh-Robinson} established the \(C^1\)-closing lemma and the $C^1$-generic density for symplectic and volume-preserving diffeomorphisms on compact manifolds. Generalizing these results to higher orders of smoothness (for maps not necessarily symplectic or volume-preserving) is the subject of Smale's tenth question \cite{Samle}. In dimension two, several recent breakthroughs have been made Floer-theoretic techniques. Using Embedded Contact Homology (ECH), Asaoka and Irie \cite{Asaoka-Irie} show that the \(C^{\infty}\)-closing lemma holds for Hamiltonian diffeomorphisms on closed surfaces. Through developing a quantitative theory of the closely related Periodic Floer Homology (PFH), Cristofaro-Gardiner, Prasad, and Zhang \cite{CGPZ} extend Asaoka and Irie's result to all area-preserving maps on closed surfaces. A key step in their proof is to show that the \( C^{\infty} \)-closing lemma holds within the Hamiltonian isotopy class of a rational area-preserving diffeomorphism (see Section \ref{rational maps} for the definition). Assuming that such a class satisfies a PFH theoretic ``\( U \)-cycle property'', Edtmair and Hutchings \cite{EH} independently provides a quantitative version of this result, constructing a Hamiltonian perturbation that creates a periodic orbit of size \( O(\delta^{-1}) \) within time \( \delta \). Later, Cristofaro-Gardiner, Pomerleano, Prasad, and Zhang \cite{CGPPZ} showed that every rational Hamiltonian isotopy class on a closed surface satisfies the \( U \)-cycle property. Combining these results, one obtains a quantitative \( C^{\infty} \)-closing lemma for rational area-preserving maps on closed surfaces. This turns out to play a crucial role in our proof of Theorem \ref{main thm1}.

Subsequently, Prasad \cite{Prasad} showed that a \( C^{\infty} \)-generic area-preserving diffeomorphism on a closed surface has an equidistributed sequence of orbit sets. Prasad's theorem serves as a quantitative refinement of the closing lemma and was later generalized by Pirnapasov and Prasad \cite{PP} to area-preserving maps on compact surfaces. The proofs of these results also heavily rely on the quantitative theory of PFH.

\subsection{The PFH Weyl law for rational area-preserving homeomorphisms}

A key result in the quantitative theory of PFH is the PFH Weyl law \cite[Theorem 1.5]{CGPZ}, which describes the asymptotic behavior of PFH spectral invariants of a rational area-preserving diffeomorphism. These spectral invariants are numerical values derived from a version of PFH known as twisted PFH, where the chain complex is filtrated by the action values of the generators. In this paper, we formulate a PFH Weyl law for certain area-preserving homeomorphisms on closed surfaces (equation (\ref{Weyl law for homeomorphisms})). If this law holds, it has potential applications in studying the quantitative dynamics of area-preserving homeomorphisms. While we establish a variant of this law when \( \Sigma = S^2 \) (Theorem \ref{double limit Weyl law}), our method is unlikely to extend to surfaces of higher genus. As such, establishing or disproving this law is of independent interest and an intriguing direction for future study.

\subsection{Proof outline}

PFH is constructed for area-preserving diffeomorphisms on closed surfaces. To adapt PFH methods to our setting, we work with diffeomorphisms on \( \Sigma \) with \( P \) as the set of singularities, where they are only continuous. Indeed, for any \( \phi \in \Diff(\Sigma_{P}, \omega_0) \), one can extend it to an area-preserving homeomorphism \( \Phi \) of \( (\Sigma, \omega) \) with \( \Phi(P) = P \). It is well-known that \( \Phi \) is the \( C^0 \)-limit of a sequence \( \{\phi_n\}_{n \geq 1} \) in \( \Diff(\Sigma, \omega) \). The challenge, then, is to show that analogous results for closed surfaces hold under this limit. For Theorem \ref{main thm1}, we achieve this by first proving a finer approximation of $\Phi$ (Lemma \ref{fine approximation}), which shows that the sequence \( \{\phi_n\}_{n \geq 1} \) can be arranged to lie in the same Hamiltonian isotopy class. We then introduce the notion of a rational area-preserving homeomorphism and show, using the quantitative closing lemma from \cite{EH}, that the \( C^{\infty} \)-closing lemma holds for such homeomorphisms. For Theorem \ref{thm2}, we apply the PFH Weyl law for rational area-preserving homeomorphisms. We establish a variant of this law when $\Sigma = S^2$ (Theorem \ref{double limit Weyl law}) and combine it with a formal argument from \cite{Prasad}.

\subsection{Outline}

The paper is organized as follows. In Section 2, we define the relevant terminologies and collect some basic facts about PFH. In Section 3, we prove Lemma \ref{fine approximation} and Theorem \ref{main thm1}. In Section 4, we formulate the PFH Weyl law of area-preserving homeomorphisms and prove Theorem \ref{double limit Weyl law}. In section 5, we prove some technical transversality results which are needed in the proof of Theorem \ref{thm2}. In Section 6, we prove analogous results to \cite[Proposition 4.1 and 4.2]{Prasad} that produce ``nearly equidistributed" orbit sets of perturbations of $\Phi$. We then use these results to prove Theorem \ref{thm2}. 

\subsection{Acknowledgements}

I want to thank my advisor, Dan Cristofaro-Gardiner, for proposing this problem and giving me his extensive support. Given that this is my first project, Dan provided me with excellent advice, both mathematically and personally.  I am grateful to Rohil Prasad and Abror Pirnapasov for reviewing my preliminary results and offering their insights. I also want to thank Guanheng Chen, Kei Irie, and Boyu Zhang for answering my questions regarding their paper.

\end{section}

\begin{section}{Preliminaries}
\subsection{Orbits and orbit sets}
Let $\phi \in \Diff(\Sigma_{P}, \omega_{0})$. A \textit{periodic orbit} of $\phi$ is a finite set $S = \{x_1,\dots,x_d\} $ such that $\phi(x_i) = x_{i+1}$ for $i \in \{1,\dots,d-1\}$, and $\phi(x_d) = x_1$. We say $S$ is \textit{simple} if all $x_i$'s are distinct, and \textit{non-degenerate} if the linearized map \( D\phi^d|_{x_1}: T_{x_1}\Sigma \to T_{x_1}\Sigma \) does not have 1 as an eigenvalue. The map \( \phi \) is $d$-\textit{non-degenerate} if all its periodic orbits of period $\leq d$ are non-degenerate. We say $\phi$ is non-degenerate if it is $d$-non-degenerate for all $d \geq 1$. 

An \textit{orbit set} $\mathcal{O}$ of $\phi$ is a formal sum $\mathcal{O} = \sum_{k = 1}^{n}a_kS_k$, where $a_k \in \R_{>0}$ and each $S_k$ is an simple periodic orbit of $\phi$. 
Its cardinality is given by $\lvert \mathcal{O} \rvert = \sum_{k = 1}^{n}a_k \lvert S_k \rvert \in \R$, where $\lv S_k \rv$ is the cardinality of an orbit. We say $\mathcal{O}$ is \textit{integral} if each $a_k \in \Z_{> 0}$. Denote the collection of orbit sets of $\phi$  by $\mathcal{P}_{\R}(\phi)$, and its integral subset by $\mathcal{P}_{\Z}(\phi)$. 

\subsubsection{Equidistribution of orbit sets}
Fix $\phi \in \Diff(\Sigma_{P},\omega)$. Let \( C_{0}^{\infty}(\Sigma_{P}) \) be the space of smooth, compactly supported functions on $\Sigma_{P}$. For $\O = \sum_{k=1}^{n}a_kS_k \in \mathcal{P}_{\R}(\phi)$, and $f \in C_{0}^{\infty}(\Sigma_{P})$, we define an action
\begin{equation*}
    \mathcal{O}(f) = \sum_{k=1}^{n}a_k S_k(f)
\end{equation*}
where $S(f) = \sum_{i=1}^{d}f(x_i)$ is the sum of $f$ evaluated over the points in $S$. A sequence of orbit sets $\{\O_N\}_{N \geq 1}$ in $\P_{\R}(\phi)$ is called \textit{equidistributed} if for all $f \in C_{0}^{\infty}(\Sigma_{P})$, 
\begin{equation}
\label{equidistribution}
\lim_{N \to \infty} \frac{\O_N(f)}{\lvert \O_N \rvert} = \int_{\Sigma}f\omega
\end{equation}
Note that if $\phi$ admits an equidistributed sequence of orbit sets, then it has a dense set of periodic points in $\Sigma_{P}$.

\subsection{Hamiltonians}

Given a Hamiltonian $H \in C^{\infty}(\R/\Z \times \Sigma)$, its \textit{Hamiltonian vector field} $X_H$ is given by
\begin{equation*}
\omega((X_H)_t, \cdot) = dH_t.
\end{equation*}
We denote by $\varphi_{H}^t$ the time-$t$ flow of $X_H$.  For two Hamiltonians \( H,K \in C^{\infty}(\R/\Z \times \Sigma) \), define their \textit{composition} as
\begin{equation*}
    (H \# K)(t,x) = H(t,x) + K(t,(\varphi_{H}^{t})^{-1}(x))
\end{equation*}
then \( \varphi_{H \# K}^{t} = \varphi_{H}^{t} \circ \varphi_{K}^{t} \) for all $t \in [0,1]$. We say two maps \(\phi, \phi' \in \Diff(\Sigma, \omega)\) are \textit{Hamiltonian isotopic} (or belong to the same Hamiltonian isotopy class) if \(\phi' = \phi \circ \varphi_H^1\) for some \(H \in C^{\infty}(\mathbb{R}/\mathbb{Z} \times \Sigma)\). In this situation, we denote \(\phi' = \phi^H\). 

We consider the \(C^{\infty}\)-topology on the space of Hamiltonians, which is induced by a complete metric \(d_{C^{\infty}}\) given as follows. Choose an auxiliary Riemannian metric \(g\) on \(\mathbb{R}/\mathbb{Z} \times \Sigma\). For \(H \in C^{\infty}(\mathbb{R}/\mathbb{Z} \times \Sigma)\), let \(\| H \|_{C^l} = \sum_{k=0}^{l} \sup |\nabla^{k} H|_{g}\) be the \(C^l\)-norm of \(H\). Then \(d_{C^{\infty}}(H, K) = \sum_{l=0}^{\infty} 2^{-l} \frac{\| H - K \|_{C^l}}{1 + \| H - K \|_{C^l}}\). We say that \(H \in C^{\infty}(\mathbb{R}/\mathbb{Z} \times \Sigma)\) is \(C^{\infty}\)-small if \(d_{C^{\infty}}(H, 0)\) is small.

\subsubsection{Hamiltonian flows preserving the punctures}

We consider some special Hamiltonians on \(\Sigma\) whose flows fix \(P\) pointwise for all time. Let 
\begin{equation*}
C^{\infty}(\R/\Z \times \Sigma ;P) = \{H \in C^{\infty}(\R/\Z \times \Sigma) \mid dH_{(t,p)} = 0 \text{ for all $t \in \R/\Z$ and $p \in P$}\}
\end{equation*}
Then $C^{\infty}(\R/\Z \times \Sigma ;P)$ is $C^{\infty}$-closed in $C^{\infty}(\R/\Z \times \Sigma)$. For each positive integer \(N \geq 1\), let \(C^{\infty}([0,1]^N \times \R/\Z \times \Sigma ;P)\) be the collection of ``multi-parameter'' Hamiltonians \(H \in C^{\infty}([0,1]^N \times \R/\Z \times \Sigma)\) such that \(H^{\tau} \in C^{\infty}(\R/\Z \times \Sigma ;P)\) for each \(\tau \in [0,1]^N\). Here, \(H^{\tau}\) is the restriction of \(H\) to \(\{\tau\} \times \R/\Z \times \Sigma\).

\subsection{The $C^0$-topology}

Let \(\Homeo(\Sigma)\) be the group of homeomorphisms of \(\Sigma\). For two elements \(\phi, \psi \in \Homeo(\Sigma)\), define the \textit{\(C^0\)-distance} between them by
\begin{equation*}
   d_{C^0}(\phi, \psi) = \max_{x \in \Sigma} d(\phi(x), \psi(x))
\end{equation*}
where \(d\) is a Riemannian distance on \(\Sigma\). Also, let
\begin{equation*}
   \overline{d}(\phi, \psi) = \max\{d_{C^0}(\phi, \psi), d_{C^0}(\phi^{-1}, \psi^{-1})\}.
\end{equation*}
Then \(\overline{d}\) defines a complete metric on \(\Homeo(\Sigma)\). The topology induced by \(\overline{d}\), called the \textit{\(C^0\)-topology}, coincides with the compact-open topology and is independent of the choice of \(d\). If \(\{\phi_n\}_{n \geq 1}\) is a Cauchy sequence that converges to \(\Phi\) in the \(C^0\)-topology, we write \(\Phi = \lim_{C^0}\phi_n\). 

\subsubsection{The mass flow homomorphism}
Let \(\Homeo(\Sigma, \omega)\) be the group of area-preserving homeomorphisms of \((\Sigma, \omega)\). It is a \(C^0\)-closed subgroup of \(\Homeo(\Sigma)\). Let \(\Homeo_0(\Sigma, \omega)\) be the identity component of \(\Homeo(\Sigma, \omega)\). In \cite{Fathi}, Fathi constructs the \textit{mass flow homomorphism}
\begin{equation}
\label{mass flow}
\theta: \Homeo_0(\Sigma, \omega) \to H_1(\Sigma; \mathbb{R}) / \Gamma_\omega
\end{equation}
where \(\Gamma_\omega = H_1(\Sigma; \mathbb{Z})\) when \(\Sigma = \mathbb{T}^2\) and \(\Gamma_\omega = 0\) otherwise. Fathi shows that \(\theta\) is surjective and \(C^0\)-continuous. Let \(\Ham(\Sigma, \omega)\) be the group of Hamiltonian diffeomorphisms of \((\Sigma, \omega)\). Then \(\ker(\theta) = \overline{\Ham(\Sigma, \omega)}\), the \(C^0\)-closure of \(\Ham(\Sigma, \omega)\) within \(\Homeo_0(\Sigma, \omega)\). The restriction of \(\theta\) to \(\Diff_0(\Sigma, \omega)\) is Poincaré dual to the flux homomorphism.

\subsection{Periodic Floer homology}
\subsubsection{The mapping torus}
Given \(\phi \in \Diff(\Sigma, \omega)\), define its mapping torus as
\begin{equation*}
\label{mapping torus}
    M_{\phi} = [0,1] \times \Sigma / (0, \phi(p)) \sim (1, p)
\end{equation*}
It admits a vector field \(R = \partial_t\) (the \textit{Reeb vector field}) and a two-form \(\omega_{\phi}\), which is the pullback of the area form \(\omega\). By construction, the periodic orbits of \(\phi\) correspond one-to-one to the orbits of \(R\) (the \textit{Reeb orbits}).

\subsubsection{Rational maps}
\label{rational maps}

Let \(\phi \in \Diff(\Sigma,\omega)\). We say that \(\phi\) is \textit{rational} if \([\omega_{\phi}] \in H^2(M_{\phi};\R)\) is a real multiple of a class in \(H^2(M_{\phi};\Z)\). Since we assume \(\int_{\Sigma} \omega = 1\), this is equivalent to saying that \([\omega_{\phi}] \in H^2(M_{\phi};\Q)\). For a Hamiltonian \(H \in C^{\infty}(\R/\Z \times \Sigma)\), the map
\begin{align}
\label{mapping tori diffeo}
    \Psi_H: [0,1] \times \Sigma &\to [0,1] \times \Sigma \\
    (t,x) &\mapsto (t,(\varphi_{H}^t)^{-1}(x)) \nonumber
\end{align}
descends to a diffeomorphism \(M_{\phi} \to M_{\phi^H}\). Note that \(\Psi_{H}^{*}[\omega_{\phi^H}] = [\omega_{\phi}]\), so \(\phi\) is rational if and only if \(\phi^H\) is rational. By virtue of this, we say that a Hamiltonian isotopy class in \(\Diff(\Sigma, \omega)\) is \textit{rational} if any map in that class is rational.

\subsubsection{Generators, the chain complex, and formal properties of PFH}
We introduce a version of Periodic Floer Homology called twisted PFH. Fix a rational, non-degenerate $\phi \in \Diff(\Sigma,\omega)$. Fix $\Gamma \in H_1(M_{\phi};\Z)$ such that for some constant $\rho \in \R$, 
\begin{equation}
\label{rationality}
    c_1(E) + 2PD(\Gamma) = -\rho [\omega_{\phi}]
\end{equation}
Here, $E$ is the vertical tangent bundle of $M_{\phi}$ (transverse to the Reeb vector field), and $PD(\Gamma)$ is the Poincar\'e dual of $\Gamma$. Let \(d(\Gamma) = \Gamma \cdot [\Sigma]\) be the \textit{degree} of the class \(\Gamma\). Since \(\phi\) is rational, there exists \(\Gamma \in H_1(M_{\phi};\Z)\) satisfying (\ref{rationality}) with arbitrarily large degree.

Assume that \(d(\Gamma) > \max(1,g)\), where \(g\) is the genus of \(\Sigma\). An (integral) orbit set of \(R\) is a formal sum \(\alpha = \sum_{i} m_i \alpha_i\), where \(m_i \in \Z_{>0}\) and the \(\alpha_i\)'s are distinct, embedded Reeb orbits. We say $\alpha$ is a \textit{PFH generator in the class $\Gamma$} if
\begin{itemize}
    \item \(m_i = 1\) whenever \(\alpha_i\) is hyperbolic.
    \item \([\alpha] = \sum_{i} m_i [\alpha_i] = \Gamma\).
\end{itemize}
Fix a 1-cycle $\gamma \subset M_{\phi}$ with $[\gamma] = \Gamma$. We call such a \(\gamma\) a \textit{reference cycle}, with degree \(d(\gamma) = d(\Gamma)\). The twisted PFH chain complex \(\text{TwPFC}(\phi,\gamma)\) a $\Z_2$-vector space over the pairs $(\alpha, Z)$, where $\alpha$ is a PFH generator in the class $\Gamma$, and $Z$ is a relative class in $H_2(M,\alpha, \gamma;\Z)$. That is, $Z$ is an equivalence class of two chains $W$ with $\partial W = \alpha- \gamma$, modulo the boundary of 3-chains.

Upon choosing an admissible almost complex structure \(J\) on \(\R \times M_{\phi}\), the twisted PFH differential counts \(J\)-holomorphic curves in \(\R \times M_{\phi}\) with ``ECH index'' 1. The resulting homology is independent of the choice of \(J\) and is denoted by \(\Tw(\phi,\gamma)\). $\Tw(\phi,\gamma)$ admits a relative $\Z$-grading, which can be made absolute once we choose a trivialization of $E$ over $\gamma$. If \(d(\gamma)\) is sufficiently large, then \(\Tw(\phi,\gamma) \neq 0\) \cite[Theorem 1.4]{CGPZ}. Twisted PFH is also invariant under Hamiltonian isotopies: Given \(H \in C^{\infty}(\R/\Z \times \Sigma)\), there is a canonical isomorphism between \(\Tw(\phi,\gamma)\) and \(\Tw(\phi^H,\gamma^H)\), where \(\gamma^H = \Psi_H(\gamma)\) and \(\Psi_H: M_{\phi} \to M_{\phi^H}\) is the diffeomorphism defined in (\ref{mapping tori diffeo}). This allows one to define the PFH of a degenerate map \(\phi\) as the PFH of its non-degenerate Hamiltonian perturbation.

\subsubsection{The $U$ map}
There is a map $U: \Tw_{*}(\phi,\gamma) \to \Tw_{*-2}(\phi,\gamma)$ which, on the chain level, counts $J$-holomorphic curves of ECH index 2 that pass through a base point $y$ in $\R \times M_{\phi}$. The $U$ map does not depend on the choice of $y$ or the almost complex structure $J$. It is also invariant under Hamiltonian isotopies, in that it commutes with the canonical isomorphism between \(\Tw(\phi,\gamma)\) and \(\Tw(\phi^H,\gamma^H)\) \cite[Proposition 3.1]{EH}. For more properties of PFH as a $U$-module, we refer readers to \cite{CGPPZ}.

\subsubsection{The spectral invariants \label{spectral invariants}}
We define an action on a generator $(\alpha, Z)$ of $\text{TwPFC}(\phi,\gamma)$ by 
\begin{equation*}
\mathcal{A}(\alpha,Z) = \int_{Z}\omega_{\phi}
\end{equation*}
Let \( \text{TwPFC}^{L}(\phi, \gamma) \) be the submodule generated by the pairs $(\alpha,Z)$ for which $\mathcal{A}(\alpha,Z) \leq L$. Given a choice of admissible almost complex structure \(J\) on \(\R \times M_{\phi}\), \(\text{TwPFC}^{L}(\phi, \gamma)\) is a subcomplex of \(\text{TwPFC}(\phi, \gamma)\). Let $\text{TwPFH}^{L}_{*}(\phi,\gamma)$ denote the resulting homology and
\begin{equation*}
    \iota_{L}: \text{TwPFH}^{L}_{*}(\phi, \gamma) \to \text{TwPFH}_{*}(\phi, \gamma)
\end{equation*}
be the map induced from the inclusion of chain complexes. For a non-zero homology class $\sigma \in \text{TwPFH}_{*}(\phi,\gamma)$, we define the \textit{spectral invariant} $c_{\sigma}(\phi,\gamma)$ to be
\begin{equation*}
c_{\sigma}(\phi,\gamma) = \inf\{L \mid \sigma \in \text{Im}(\iota_{L})\}
\end{equation*}
We state two important properties of PFH spectral invariants. The first property allows us to extend PFH spectral invariants to the degenerate case (see \cite[Section 2.2.3]{CGPZ}). The extended PFH spectral invariants will satisfy the same properties. 

\begin{prop}{(Hofer-Lipschitz continuity, \cite[Proposition 5.1]{CGPZ})}
Fix a reference cycle \(\gamma\) in \(M_{\phi}\) with degree \(d\) and a non-zero class \(\sigma \in \Tw(\phi,\gamma)\). Let \(H, K \in C^{\infty}(\R/\Z \times \Sigma)\) be Hamiltonians such that both \(\phi^{H}\) and \(\phi^{K}\) are non-degenerate. Then, we have
\begin{equation}
\label{hofer-lipschitz}
d\int_{0}^{1}\min (H_t - K_t)\,dt \leq c_{\sigma}(\phi^H,\gamma^{H}) - c_{\sigma}(\phi^K,\gamma^{K}) + \int_{\gamma} (H - K)\,dt \leq d\int_{0}^{1}\max (H_t - K_t)\,dt
\end{equation}
where \(\sigma\) is viewed as a class in \(\Tw(\phi^{H},\gamma^{H})\) or \(\Tw(\phi^{K},\gamma^{K})\) under the canonical isomorphism with \(\Tw(\phi,\gamma)\).
\end{prop}

\begin{prop}{(The PFH Weyl law, \cite[Theorem 1.5]{CGPZ})}
Let \(\{\gamma_m\}_{m \geq 1}\) be a sequence of reference cycles in \(M_{\phi}\) with increasing degrees \(d_m\) tending to \(\infty\), and \(\{\sigma_m\}_{m \geq 1}\) a sequence of non-zero classes in \(\Tw(\phi,\gamma_m)\). For \(H \in C^{\infty}(\R/\Z \times \Sigma)\), we have
\begin{equation}
\label{CGPZ weyl law}
   \lim_{m \to \infty} \frac{c_{\sigma_m}(\phi^{H},\gamma_m^{H}) - c_{\sigma_m}(\phi,\gamma_m) + \int_{\gamma_m}H \, dt}{d_m} = \int_{[0,1] \times \Sigma}H \,\omega \wedge dt
\end{equation}
\end{prop}

\subsubsection{Comparison with other Floer homologies}
Let \(\underline{L}\) be a link on \(\Sigma\). In \cite{Chen}, Chen constructed a Floer homology \(HF(\Sigma, \underline{L})\), which has associated numerical invariants called \textit{HF spectral invariants}. Formally, these are maps
\begin{equation}
\label{HF spectral invariants}
    c_{\underline{L}}: C^{\infty}(\R/\Z \times \Sigma) \times HF(\Sigma, \underline{L}) \to \{-\infty\} \cup \mathbb{R}.
\end{equation}
Chen also constructed an \textit{open-closed morphism} between HF and the twisted PFH, which allows for the comparison of PFH spectral invariants with HF spectral invariants. We will use these results in Section 4. 
\end{section}

\section{Proof of Theorem \ref{main thm1}}
Let $\Phi \in \Homeo(\Sigma,\omega)$ be an area-preserving homeomorphism. According to \cite[Theorem I]{Oh}, $\Phi$ is the $C^0$-limit of a sequence $\{\phi_n\}_{n \geq 1}$ in $\Diff(\Sigma,\omega)$. We now prove a refinement of this result. 
\begin{lem}
\label{fine approximation}
Let \(\Phi \in \Homeo(\Sigma, \omega)\). Then \(\Phi = \lim_{C^0} \phi_n\) for some sequence \(\{\phi_n\}_{n \geq 1}\) in \(\Diff(\Sigma, \omega)\) which are all Hamiltonian isotopic. In other words, \(\Phi\) lies in the \(C^0\)-closure of a Hamiltonian isotopy class in $\Diff(\Sigma,\omega)$. In addition, this Hamiltonian isotopy class is unique.
\end{lem}
\begin{proof}
Let \(\P(\Phi)\) be the path component of \(\Phi\) in \(\Homeo(\Sigma, \omega)\). Since \(\Homeo(\Sigma, \omega)\) is locally contractible, \(\P(\Phi)\) is open. As \(\Diff(\Sigma, \omega)\) is \(C^0\)-dense in \(\Homeo(\Sigma, \omega)\), there exists \(\varphi \in \Diff(\Sigma, \omega) \cap \P(\Phi)\). Thus, \(\Phi\) and \(\varphi\) are isotopic as area-preserving homeomorphisms, and \(\varphi^{-1} \circ \Phi \in \Homeo_{0}(\Sigma, \omega)\). 

Let \(\theta\) be the mass-flow homomorphism in (\ref{mass flow}). Choose \(\psi \in \Diff_{0}(\Sigma, \omega)\) such that \(\theta(\psi) = \theta(\varphi^{-1} \circ \Phi)\), then \(\theta(\psi^{-1} \circ \varphi^{-1} \circ \Phi) = 0\). Since \(\ker(\theta) = \overline{\Ham(\Sigma, \omega)}\), it follows that \(\psi^{-1} \circ \varphi^{-1} \circ \Phi = \lim_{C^0} \varphi_{H_n}^{1}\) for some \(\{H_n\}_{n \geq 1}\) in \(C^{\infty}(\mathbb{R}/\mathbb{Z} \times \Sigma)\). Thus, \(\Phi = \lim_{C^0} \phi_n\) where \(\phi_n = (\varphi \circ \psi)^{H_n}\).

To prove uniqueness, assume that \(\Phi = \lim_{C^0}\varphi_{1}^{H_n} = \lim_{C^0}\varphi_{2}^{K_n}\) where \(\varphi_1, \varphi_2 \in \Diff(\Sigma, \omega)\). Then \(\text{Id}_{\Sigma} = \lim_{C^0} \varphi_{1}^{H_n} \circ (\varphi_{2}^{K_n})^{-1}\). For each fixed \(n\), we have
\begin{equation*}
\theta[\varphi_{1}^{H_n} \circ (\varphi_{2}^{K_n})^{-1}] = \theta(\varphi_1 \circ \varphi_{H_n \# \overline{K_n}}^{1} \circ \varphi_2^{-1}) = \theta(\varphi_1 \circ \varphi_{2}^{-1})
\end{equation*}
where \(\overline{K_n}(t, x) = -K_n(t, (\varphi_{K_n}^t)(x))\). The \(C^0\)-continuity of $\phi$ implies that \(\theta(\varphi_1 \circ \varphi_2^{-1}) = 0\). Thus, \(\varphi_1\) and \(\varphi_2\) are Hamiltonian isotopic.
\end{proof}

\begin{rem}
For $\phi \in \Diff(\Sigma,\omega)$, denote its Hamiltonian isotopy class as the coset $\phi\Ham(\Sigma, \omega)$. Lemma \ref{fine approximation} implies that $\Homeo(\Sigma,\omega) = \bigsqcup_{\phi \in \Diff(\Sigma,\omega)} \overline{\phi \Ham(\Sigma,\omega)}$, where $\overline{\phi \Ham(\Sigma,\omega)}$ is the $C^0$-closure of $\phi\Ham(\Sigma,\omega)$ within $\Homeo(\Sigma,\omega)$. 
\end{rem}
By virtue of Lemma 3.1, we say an area-preserving homeomorphism \(\Phi \in \Homeo(\Sigma, \omega)\) is \textit{rational} if it lies in the \(C^0\)-closure of a rational Hamiltonian isotopy class.

\subsection{Proof of Theorem \ref{main thm1}}

We prove the theorem in the following steps.  

\subsubsection{Setting}

Let \(\Phi \in \Homeo(\Sigma, \omega)\) be the extension of \(\phi\) such that \(\Phi(P) = P\). By shrinking \(U\) if necessary, we may assume that \(\overline{U} \cap P = \emptyset\). 

\subsubsection{Perturbation of \(\Phi\) to a rational map}

By Lemma \ref{fine approximation}, \(\Phi = \lim_{C^0} \phi^{H_n}\) for some \(\phi \in \Diff(\Sigma, \omega)\) and \(\{H_n\}_{n \geq 1}\) in \(C^{\infty}(\R/\Z \times \Sigma)\). By \cite[Lemma 5.4]{CGPZ}, there exists a symplectic vector field \(X\) on $\Sigma$, compactly supported on $\Sigma_P$, such that \(\varphi_X^{1} \circ \phi\) is rational. It follows that \(\Phi' = \lim_{C^0} \phi_n\) is rational, where $\Phi' = \varphi_{X}^1 \circ \Phi$ and \(\phi_n = (\varphi_X^{1} \circ \phi)^{H_n}\). 

\subsubsection{Applying the quantitative closing lemma} 

Let \( H \) be a \((U, a, l)\)-admissible Hamiltonian as defined in \cite[Definition 1.12]{EH} (\(a\) and \(l\) are positive constants). According to \cite[Corollary 2]{CGPPZ} and \cite[Theorem 7.4]{EH}, for each \(n\), there exists \(\tau_n \in [0, a/l]\) such that \(\phi_n \circ \varphi_{H}^{\tau_n}\) has a periodic orbit \(S_n\) passing through \(U\), with \(\lvert S_n \rvert \leq d_n\). Here, \(d_n\) depends on \(a\), \(l\), and the Hamiltonian isotopy class of \(\phi_n\). Since all \(\phi_n\) are Hamiltonian isotopic, we have \(d_n = d(a,l)\). By choosing subsequences, we assume that \(\tau_n \to \tau \in [0, a/l]\), and that \(S_n\) all have the same cardinality. It follows that a subsequence of \(\{S_n\}_{n \geq 1}\) converges to a periodic orbit \(S\) of \(\Phi' \circ \varphi_{H}^{\tau}\) that passes through \(\overline{U}\). Let \(\phi' =  \varphi_{X}^{1}|_{\Sigma_P} \circ \phi \circ \varphi_{H}^{\tau}|_{\Sigma_P}\). Since \(\phi'\) and \(\phi\) coincide outside a compact subset of \(\Sigma_P\), we can take \(X\) and \(H\) to be \(C^{\infty}\)-small so that \(\phi' \in V\). 

\begin{rem}
Rational area-preserving homeomorphisms can also be characterized using algebraic topological methods. In \cite{Schwartzman}, Schwartzman introduced the asymptotic cycle of a measure-preserving flow on a compact manifold. For $\Phi \in \Homeo(\Sigma, \omega)$, we denote by $\cC(\Phi) \in H_1(M_{\Phi}; \R)$ the asymptotic cycle associated with the suspension flow on the mapping torus $M_{\Phi}$, which preserves the measure $dt \otimes \text{Area}(\Sigma)^{-1}\omega = dt \otimes \omega$. In \cite{Prasad2}, Prasad established dynamical properties of maps with a rational asymptotic cycle, i.e., $\cC(\Phi) \in H_1(M_{\Phi}; \Q)$. When $\Phi$ is isotopic to the identity (in which case $\Phi$ is said to have a rational rotation vector), similar results were established independently by Guih\'eneuf-Le Calvez- Passeggi \cite{GLCP} using topological methods. In particular, \cite[Lemma 3.2]{Prasad2} implies that an area-preserving diffeomorphism is rational if and only if it has a rational asymptotic cycle. It is not difficult to show that the same statement holds for area-preserving homeomorphisms.
\end{rem}

\section{The PFH Weyl law for rational area-preserving homeomorphisms}
\label{Weyl law section}

We use Lemma \ref{fine approximation} to formulate the PFH Weyl law for rational area-preserving homeomorphisms on closed surfaces. This provides a potentially useful tool for studying surface dynamics in a continuous setting.

\subsection{Formulation}
\label{formulation}
Let \( \Phi \in \Homeo(\Sigma, \omega) \) be a rational area-preserving homeomorphism. By Lemma \ref{fine approximation}, we can write \( \Phi = \lim_{C^0} \phi^{H_n} \), where \( \phi \in \Diff(\Sigma, \omega) \) is rational. Let \( \{\gamma_m\}_{m \geq 1} \subset M_{\phi} \) be a sequence of reference cycles with degrees \( d(\gamma_m) \) tending to infinity. For each \( m \geq 1 \), choose a non-zero class \( \sigma_m \in \Tw(\phi, \gamma_m) \). We are interested in the following question.

\begin{quest}
\label{main question}
Given $K \in C^{\infty}(\R/\Z \times \Sigma$). Is it true that
\begin{equation}
\label{Weyl law for homeomorphisms}
    \lim_{m \to \infty} \left(\lim_{n \to \infty} \frac{c_{\sigma_m}(\phi^{H_n \# K
    },\gamma_{m}^{H_n \# K}) - c_{\sigma_m}(\phi^{H_n},\gamma_m^{H_n}) + \int_{\gamma_m^{H_n}}K \,dt}{d_m} \right) = \int_{0}^{1}\int_{\Sigma} K \, \omega \wedge dt \hspace{1mm}?
\end{equation}
\end{quest}

Note that if we exchange the two limits in equation (\ref{Weyl law for homeomorphisms}), the equation would hold by (\ref{CGPZ weyl law}). Therefore, Question \ref{main question} can be broken down to the following two questions.

\begin{quest}
\label{inner limit}
Given $\gamma \subset M_{\phi}$, $\sigma \in \Tw(\phi,\gamma)$ and $K \in C^{\infty}(\R/\Z \times \Sigma)$. Does the $C^0$-convergence $\{\phi^{H_n}\}_{n \geq 1} $ imply that the existence of the limit
    \begin{equation*}
        \lim_{n \to \infty} c_{\sigma}(\phi^{H_n \# K}, \gamma^{H_n \# K}) + \int_{\gamma^{H_n}} K \, dt \hspace{1mm}?
    \end{equation*}
\end{quest}

\begin{quest}
\label{exchange limits}
Are the inner and outer limits in (\ref{Weyl law for homeomorphisms}) interchangeable? 
\end{quest}

We are interested in answering Question \ref{inner limit} and Question \ref{exchange limits} in full generality. In the next section, we provide partial answers to these questions and prove a variant of (\ref{Weyl law for homeomorphisms}) when \( \Sigma = S^2 \). Specifically, we take \( \phi = \text{Id}_{S^2} \) and an arbitrary sequence of Hamiltonians \( \{H_n\}_{n \geq 1} \) in \( C^{\infty}(\mathbb{R}/\mathbb{Z} \times S^2) \). Since there is no requirement for the \( C^0 \)-convergence of \( \{\phi^{H_n}\}_{n \geq 1} \), we replace the inner limit with \(\limsup\). Additionally, $\{\gamma_m\}_{m \geq 1}$ and \( \{\sigma_m\}_{m \geq 1} \) are special choices of PFH data, which we explain below. 

\subsection{A variant of (\ref{Weyl law for homeomorphisms})}
Let \( \Sigma \) be a closed surface of genus \( g \). Fix \( x \in \Sigma \). Let \(\gamma = S^1 \times \{x\} \subset S^1 \times \Sigma = M_{\text{Id}_{\Sigma}}\) be a reference cycle of degree 1. For each $d > g$, let \(\sigma_{\heartsuit}^d \in \Tw(\text{Id}_{\Sigma}, d \cdot \gamma)\) be the \textit{PFH unit class} defined in \cite[Section 7]{Chen2}. We will review its definition in Section \ref{PFH unit class}. 

Given \(H \in C^{\infty}(\R/\Z \times \Sigma)\), write \(c_d(H) = c_{\sigma_{\heartsuit}^d}(\varphi_{H}^{1}, d \cdot \gamma^H)\). When $\Sigma = S^2$, we establish the following variant of  (\ref{Weyl law for homeomorphisms}). 

\begin{thm}
\label{double limit Weyl law}
Let $\Sigma = S^2$. Let $K$ and $\{H_n\}_{n \geq 1}$ be Hamiltonians in $C^{\infty}(\R/\Z \times S^2)$. We have
\begin{equation}
\label{the weyl law}
\lim_{d \to \infty}\left(\limsup_{n \to \infty}\frac{c_d(H_n \# K) - c_d(H_n) + \int_{d \cdot \gamma^{H_n}}K dt }{d} \right) = \int_{0}^{1} \int_{S^2} K\, \omega \wedge dt.
\end{equation}
\end{thm}

In particular, we address Question \ref{inner limit} and Question \ref{exchange limits} as follows: 
\begin{itemize}
    \item For each $d \geq 1$, the inner $\limsup$ in (\ref{the weyl law}) exists by the Hofer-Lipschitz continuity (\ref{hofer-lipschitz}).
    \item When \( \Sigma = S^2 \), the spectral invariants of the PFH unit class satisfy a triangle-type inequality. Namely, for \( d \geq 1 \) and for any two Hamiltonians \( H, K \in C^{\infty}(\mathbb{R}/\mathbb{Z} \times S^2) \), we have: 
\begin{equation}
\label{triangle inequality}
    \left| c_d(H \# K) - c_d(H) -c_d(K) - \int_{d \cdot \gamma}Kdt
    + \int_{d \cdot \gamma^{H}}K dt   \right| \leq 3
\end{equation}
Hence, one can heuristically eliminate the \( H_n \)-dependence of the term \( c_d(H_n \# K) - c_d(H_n) + \int_{d \cdot \gamma^{H_n}} K \, dt \).
\end{itemize}

\subsection{Proof of Theorem \ref{double limit Weyl law}} 

We now give a detailed proof of Theorem \ref{double limit Weyl law}. 

\subsubsection{Existence of the inner lim sup}
For each \((d, n) \in \N^{*} \times \N^{*}\), set
\begin{equation*}
    c_{d,n}(K) = \frac{c_d(H_n \# K) - c_d(H_n) + \int_{d \cdot \gamma^{H_n}} K \, dt}{d}
\end{equation*}
Then $ d \cdot c_{d,n}(K) = c_d(H_n \# K) - c_d(H_n) + \int_{d \cdot \gamma_0}(H_n \# K - H_n) \, dt$. By (\ref{hofer-lipschitz}),
\begin{equation*}
    d \cdot \int_{0}^{1} \min (H_n \# K - H_n)_t \, dt \leq d \cdot c_{d,n}(K) \leq d \cdot \int_{0}^{1} \max (H_n \# K - H_n)_t \, dt.
\end{equation*}
Thus, $ \int_{0}^{1} \min K_t \, dt \leq c_{d,n}(K) \leq \int_{0}^{1} \max K_t \, dt$, which implies \(\limsup _{n \to \infty} c_{d,n}(K) < \infty\) for each \(d \geq 1\).

\subsubsection{The outer limit}
Assume (\ref{triangle inequality}) holds. For each \(d \geq 1\), we have
\begin{equation*}
   \left\lvert d \cdot c_{d,n}(K) - c_d(K) - \int_{d \cdot \gamma} K \, dt \right\rvert \leq 3.
\end{equation*}
Thus,
\begin{equation*}
    \left\lvert \limsup_{n \to \infty} c_{d,n}(K) - \frac{c_d(K) + \int_{d \cdot \gamma} K \, dt}{d} \right\rvert \leq \frac{3}{d}.
\end{equation*}
It follows that
\begin{equation*}
     \lim_{d \to \infty} \left(\limsup_{n \to \infty} c_{d,n}(K)\right) = \lim_{d \to \infty} \frac{c_d(K) + \int_{d \cdot \gamma_0} K \, dt}{d} = \int_{[0,1] \times S^2} K \, \omega \wedge dt,
\end{equation*}
where the last equality follows from (\ref{CGPZ weyl law}). 

\subsection{Proof of (\ref{triangle inequality})}

It remains to prove the triangle-type inequality (\ref{triangle inequality}). To do this, we use the results in \cite{Chen}, which compare PFH and HF spectral invariants.

\subsubsection{Properties of the HF spectral invariants}

Let \(\underline{L} \subset S^2 \) be a 0-admissible link and \(e \in HF(S^2, \underline{L})\) be the \textit{HF unit class} (\cite[Definition 1.1 and 3.7]{Chen}). For \(H \in C^{\infty}(\R/\Z \times S^2)\), write \(c_{\underline{L}}^{+}(H) = c_{\underline{L}}(H, e)\), where \(c_{\underline{L}}\) is the map in (\ref{HF spectral invariants}). By \cite[Theorem 1 and Lemma 6.2]{Chen}, the triangle inequalities:
\begin{equation}
\label{HF triangle inequality}
    c_{\underline{L}}^{+}(H \# K) \leq c_{\underline{L}}^{+}(H) + c_{\underline{L}}^{+}(K) \leq c_{\underline{L}}^{+}(H \# K) + 1
\end{equation}
hold for any \(H, K \in C^{\infty}(\R/\Z \times S^2)\). 

\subsubsection{Comparing PFH and HF spectral invariants}

By \cite[Corollary 1.3]{Chen}, for any \(H \in C^{\infty}(\R/\Z \times S^2)\), we have
\begin{equation}
\label{comparison}
    c_d(H) + d \int_{0}^{1} H(t,x) \, dt - 1 \leq c_{\underline{L}}^{+}(H) \leq c_d(H) + d \int_{0}^{1} H(t,x) \, dt.
\end{equation}
By (\ref{comparison}) and the first inequality in (\ref{HF triangle inequality}), we have
\begin{align*}
    c_d(H \# K) \leq &c_{\underline{L}}^{+}(H \# K)  - d\int_{0}^{1}(H \# K)(t,x) \, dt + 1 \\
    \leq &c_{\underline{L}}^{+}(H) + c_{\underline{L}}^{+}(K) - d\int_{0}^{1}(H \# K)(t,x) \, dt + 1 \\
    \leq &c_d(H) + c_d(K) + d\int_{0}^{1}H(t,x) \, dt + d\int_{0}^{1}K(t,x) \, dt - d\int_{0}^{1}(H \# K)(t,x) \, dt + 1 \\
    = &c_d(H) + c_d(K) + d\int_{0}^{1} K(t,x) \, dt - K(t, (\varphi_H^t)^{-1}(x)) \, dt + 1 \\
    = &c_d(H) + c_d(K) + \int_{d \cdot \gamma} K \, dt - \int_{d \cdot \gamma^H} K \, dt + 1
\end{align*}
A similar computation using (\ref{comparison}) and the second inequality in (\ref{HF triangle inequality}) shows that
\begin{equation*}
   c_d(H \# K) \geq c_d(H) + c_d(K) + \int_{d \cdot \gamma} K \, dt - \int_{d \cdot \gamma^H} K \, dt - 3.
\end{equation*}
This finishes the proof of (\ref{triangle inequality}).

\subsection{The PFH unit class}
\label{PFH unit class}
 We recall the definition of the PFH unit class and explain why the results in \cite{Chen} imply (\ref{triangle inequality}) only when $\Sigma = S^2$. Assume that $\Sigma$ is a closed surface with genus $g$. For each $d > g$, one can choose a \( C^2 \)-small perfect Morse function \( H_{\text{Morse}}: \Sigma \to \mathbb{R} \) such that the only periodic orbits of \( \varphi_{H_{\text{Morse}}}^{1} \) with period \( \leq d \) are the critical points of \( H_{\text{Morse}} \) and their iterates. Denote the critical points by \( e_{+}, h_1, \dots, h_{2g}, e_{-} \), where \( e_{+} \) has (Morse) index 2, \( e_{-} \) has index 0, and each \( h_{1}, \dots, h_{2g} \) has index 1. It follows that that \( e_{+} \) and \( e_{-} \) are elliptic orbits, and \( h_{1}, \dots, h_{2g} \) are positive hyperbolic orbits.

If \( \alpha \) is a PFH generator in the homology class \( d \cdot [\gamma^{H_{\text{Morse}}}] \), then we can write \( \alpha = e_{+}^{m_{+}} h_1^{m_1} \dots h_{2g}^{m_{2g}} e_{-}^{m_{-}} \), where \( m_{+} + m_1 + \dots + m_{2g} + m_{-} = d \) and \( m_i \in \{0, 1\} \). A twisted PFH generator can be expressed as \( (\alpha, Z_{\alpha} + k[\Sigma] + [S]) \), where \( k \in \mathbb{Z} \), \( S \in H_1(S^1; \mathbb{Z}) \otimes H_1(\Sigma; \mathbb{Z}) \), and \( Z_{\alpha} \in H_2(M_{\varphi_{H_{\text{Morse}}}^{1}}, \alpha, d \cdot \gamma^H) \) is a reference relative homology class defined in \cite[(5.32)]{Chen}. As explained in \cite[Example 2.18]{EH}, we can choose an almost complex structure \( J \) on \( \mathbb{R} \times M_{\varphi_{H_{\text{Morse}}}^{1}} \) such that the only \( J \)-holomorphic curves that contribute to the PFH differential are \( J \)-holomorphic cylinders corresponding to the Morse flow lines on \( \Sigma \). Since \( H_{\text{Morse}} \) is perfect, it follows that the PFH differential vanishes, and each twisted PFH generator defines a homology class.

Let \( \sigma_{\heartsuit}^d = (e_{+}^d, Z_{e_{+}^d}) \) and \( \sigma_{\diamond}^d = (e_{-}^d, Z_{e_{-}^d}) \). In \cite[Corollary 1.2]{Chen}, Chen estimated the HF spectral invariants in terms of the PFH spectral invariants, with the upper bound computed using \( \sigma_{\heartsuit}^d \) and the lower bound computed using \( \sigma_{\diamond}^d \). When \( \Sigma = S^2 \), we have \( U^d \sigma_{\heartsuit}^d = \sigma_{\diamond}^d \), and Chen used this relation to deduce equation (\ref{comparison}). However, when \( \Sigma \neq S^2 \), we have no general understanding of how \( \sigma_{\heartsuit}^d \) and \( \sigma_{\diamond}^d \) relate via the \( U \)-map. Hence, equation (\ref{comparison}) might no longer hold on surfaces of higher genus.

\section{Parametric transversality}
In this section, we prove certain technical transversality results required for the proof of Theorem \ref{thm2}. The statements and proofs are largely based on \cite[Lemma 5.1–5.3]{Prasad}, with details modified to suit our setting. 

\begin{lem}
\label{parametric transversality}
Let \(\phi \in \Diff(\Sigma,\omega)\) with \(\phi(P) = P\). Fix \(N \geq 1\). For a \(C^{\infty}\)-generic \(H \in C^{\infty}([0,1]^N \times \R/\Z \times \Sigma ;P)\), there exists a full measure subset of \(\tau \in [0,1]^N\) for which \(\phi^{H^{\tau}}\) is non-degenerate.
\end{lem}

\begin{lem}
\label{non-degeneracy}
Let \(\Phi \in \Homeo(\Sigma, \omega)\) such that \(\Phi(P) = P\) and \(\Phi|_{\Sigma_P} \in \Diff(\Sigma_P, \omega_0)\). Let \(S_1, \dots, S_k\) be simple periodic orbits of \(\Phi\) which lie in \(\Sigma_P\). For \(C^{\infty}\)-dense \(H \in C^{\infty}(\R/\Z \times \Sigma; P)\), \(S_1, \dots, S_k\) are non-degenerate periodic orbits of \(\Phi^H\).
\end{lem}

\subsection{Proof of Lemma \ref{parametric transversality}}

Note that for fixed \( H \in C^{\infty}([0,1]^N \times \R/\Z \times \Sigma ;P) \) and \( \tau \in [0,1]^N \), every periodic orbit of \( \phi^{H^{\tau}} \) lies either in \( \Sigma_P \) or in \( P \). Thus, we prove Lemma \ref{parametric transversality} by combining Lemma \ref{PT outside P} and Lemma \ref{PT inside P}, which address parametric transversality inside and outside the set of punctures, respectively. 

\begin{lem}
\label{PT outside P}
Let \(\phi \in \Diff(\Sigma,\omega)\) with \(\phi(P) = P\). Fix \(N \geq 1\). For a \(C^{\infty}\)-generic \(H \in C^{\infty}([0,1]^N \times \R/\Z \times \Sigma ;P)\), there exists a full measure subset of \(\tau \in [0,1]^N\) such that every periodic orbit of \(\phi^{H^{\tau}}\) which lies in $\Sigma_P$ is non-degenerate. 
\end{lem}

\begin{lem}
\label{PT inside P}
Let \(\phi \in \Diff(\Sigma,\omega)\) with \(\phi(P) = P\). Fix \(N \geq 1\). For a \(C^{\infty}\)-generic \(H \in C^{\infty}([0,1]^N \times \R/\Z \times \Sigma ;P)\), there exists a full measure subset of \(\tau \in [0,1]^N\) such that every periodic orbit of \(\phi^{H^{\tau}}\) which lies in $P$ is non-degenerate.
\end{lem}

\subsection{Proof of Lemma \ref{PT outside P}}

\subsubsection{Notations}
For any $l \geq 3$, let 
\begin{equation*}
C^l(\R/\Z \times \Sigma ;P) = \{H \in C^{l}(\R/\Z \times \Sigma) \mid dH_{(t,p)} = 0 \text{ for all $t \in \R/\Z$ and $p \in P$}\}
\end{equation*}
Let $C^l([0,1] \times \R/\Z \times \Sigma ;P)$ be the collection of $H \in C^{l}([0,1]^N \times \R/\Z \times \Sigma)$ such that $H^{\tau} \in C^l(\R/\Z \times \Sigma ;P)$ for each $\tau \in [0,1]^N$. This is a Banach subspace under the $C^l$-norm. 

\begin{claim}
\label{PT claim}
For each $l \geq 3$, there exists a generic set of $H \in  C^l([0,1]^{N} \times \R/\Z \times \Sigma ;P)$ such that 
\begin{equation*}
    \text{measure}(\{ \tau \in [0,1]^N \mid \text{every periodic orbit of $\phi^{H^{\tau}}$ in $\Sigma_P$ is non-degenerate} \}) = 1
\end{equation*} 
\end{claim}

Lemma \ref{PT outside P} follows from Claim \ref{PT claim} by the same argument as in \cite[Section 5.1.2]{Prasad}, upon verifying that \(C^{\infty}([0,1]^N \times \R/\Z \times \Sigma ;P)\) is dense in \(C^l([0,1]^N \times \mathbb{R}/\mathbb{Z} \times \Sigma; P)\) in the \(C^l\)-topology.

\subsubsection{Proof of Claim \ref{PT claim}} 
Fix $l \geq 3$ and $d \geq 1$. Let $\Delta^d \subset {\Sigma_{P}}^d$ be the collections of tuples $(x_1,\dots,x_d)$ such that $x_i = x_j$ for some $i \neq j$. Let
\begin{align*}
    \Psi_d: C^l([0,1]^{N} \times \R/\Z \times \Sigma ;P) \times [0,1]^{N} \times ({\Sigma_{P}}^d \setminus \Delta^{d})  \to {\Sigma_{P}}^d \times {\Sigma_{P}}^d
\end{align*}
be defined by sending $(H,\tau,x_1,\dots,x_d)$ to $(x_1,\dots,x_d,\phi^{H^{\tau}}(x_d),\phi^{H^{\tau}}(x_1),\dots,\phi^{H^{\tau}}(x_{d-1}))$. 
Let $Z$ be the diagonal of ${\Sigma_{P}}^d \times {\Sigma_{P}}^d$, then $\M^{l,d} \defeq \Psi_d^{-1}(Z)$ is the collection of tuples $(H,\tau,S)$ such that $S = (x_1,\dots,x_d)$ is a simple periodic orbit of $\phi^{H^{\tau}}$ which lies in $\Sigma_P$. Let $\M^{l,d}_{bad} \subset \M^{l,d}$ be the collection of tuples $(H,\tau,S)$ such that the linearized return map 
\begin{equation*}
    \rho(S,H^{\tau}) \defeq D(\phi^{H^{\tau}})^{d}_{x_1} \in \Sp(T_{x_1}\Sigma,\omega)
\end{equation*}
at $x_1$ has a root of unity as an eigenvalue. Here, $\Sp(T_{x_1}\Sigma,\omega)$ is the group of linear symplectic automorphisms of $(T_{x_1}\Sigma,\omega)$. By repeating the arguments in \cite[Section 5.1.4 - 5.1.5]{Prasad}, one can show that
\begin{itemize}
    \item $\Psi_d$ is transverse to $Z$, $\M^{l,d}$ is a Banach manifold of class $C^{l-1}$, and the projection $\M^{l,d} \to C^l([0,1]^{N} \times \R/\Z \times \Sigma ;P)$ is a $C^{l-1}$ Fredholm map of index $N$. 
    \item $\M^{l,d}_{\text{bad}}$ is a countable union of $C^{l-2}$ Banach submanifolds of $\M^{l,d}$ of codimension at least 1. 
\end{itemize}
The claim then follows from the bullet points, as explained in \cite[Section 5.1.3]{Prasad}.

\subsection{Proof of lemma \ref{PT inside P}}

\subsubsection{}

It suffices to show that for each $l \geq 3$, there exists a generic set $\H^l_{\text{good}}(P)$
of $C^l([0,1]^{N} \times \R/\Z \times \Sigma ;P)$ such that for each $H \in \H^l_{\text{good}}(P)$, 
\begin{equation*}
    \text{measure}(\{ \tau \in [0,1]^N \mid \text{every periodic orbit of $\phi^{H^{\tau}}$ in $P$ is non-degenerate} \}) = 1
\end{equation*}

Let $S$ be a simple periodic orbit of $\phi$ in $P$. Let $\H^{l}_{\text{good}}(S)$ be the collection of $H \in C^l([0,1]^{N} \times \R/\Z \times \Sigma ;P)$ such that
\begin{equation*}
    \text{measure}(\{ \tau \in [0,1]^N \mid  \text{$S^k$ is a non-degenerate orbit of $\phi^{H^{\tau}}$ for all $k \geq 1$} \}) = 1
\end{equation*}
where $S^k$ is the $k$-th iterate of $S$. Then $\H^l_{\text{good}}(P)$ is the finite intersection
\begin{equation*}
    \H^l_{\text{good}}(P) = \bigcap_{\substack{S \subset P \\ \text{simple orbit of } \phi}} \H^{l}_{\text{good}}(S)
\end{equation*}
        It remains to show that $\H^{l}_{\text{good}}(S)$ is residual in $C^l([0,1]^{N} \times \R/\Z \times \Sigma ;P)$ for each $S$. 

\subsubsection{Proof that $\H^{l}_{\text{good}}(S)$ is residual}

Let $S = \{x_1,\dots,x_d\}$ be a simple periodic orbit of $\phi$ in $P$. Consider 
\begin{equation*}
    \rho_{S}: C^l([0,1]^{N} \times \R/\Z \times \Sigma ;P) \times [0,1]^N \to \text{Sp}(T_{x_1}\Sigma,\omega)
\end{equation*}
defined by sending $(H,\tau)$ to $\rho(S,H^{\tau})$. Let \(\cZ \subset \text{Sp}(T_{x_1}\Sigma, \omega)\) be the collection of linear symplectic automorphisms having a root of unity as an eigenvalue. Then \(\cZ\) is a union of submanifolds of \(\Sp(T_{x_1}\Sigma, \omega)\) of codimension at least 1. 

Let $\P_{\text{bad}}$ be the collection of ``bad pairs'' $(H,\tau) \in C^l([0,1]^{N} \times \R/\Z \times \Sigma ;P) \times [0,1]^N$ such that $S^k$ is a degenerate orbit of $\phi^{H^{\tau}}$ for some $k$. Then $\P_{\text{bad}} = \rho_{S}^{-1}(\cZ)$. We want to show that $\rho_{S}$ is transverse to $\cZ$, so that $\P_{\text{bad}}$ is a union of Banach submanifolds of $C^l([0,1]^{N} \times \R/\Z \times \Sigma ;P) \times [0,1]^N$ of codimension at least 1. For now, assume that this is true. Let 
\begin{align*}
    &\pi_{1}: C^l([0,1]^{N} \times \R/\Z \times \Sigma ;P) \times [0,1]^N \to  C^l([0,1]^{N} \times \R/\Z \times \Sigma ;P)\\
    &\pi_2: C^l([0,1]^{N} \times \R/\Z \times \Sigma ;P) \times [0,1]^N \to [0,1]^N
\end{align*}
be the canonical projections. By applying the Sard-Smale theorem to $\pi_{1}$ and $\pi_1|_{\P_{\text{bad}}}$, one can find a generic set $\G \subset C^l([0,1]^{N} \times \R/\Z \times \Sigma ;P)$, such that for each $H \in \G$,
\begin{itemize}
    \item $\pi_{1}^{-1}(H)$ is a submanifold of $C^l([0,1]^{N} \times \R/\Z \times \Sigma ;P) \times [0,1]^N$ of dimension $N$. 
    \item $\pi_{1}^{-1}(H) \cap \P_{\text{bad}}$ is a union of submanifolds of $\pi_{1}^{-1}(H)$ of codimension at least 1.
\end{itemize}
Fix $H \in \G$. By Sard's theorem, there exists a full measure subset $\tau(H) \subset [0,1]^N$ such that each $\tau \in \tau(H)$ is a regular value for both the maps $\pi_2|_{\pi_{1}^{-1}(H)}$ and $\pi_2|_{\pi_{1}^{-1}(H) \cap \P_{\text{bad}}}$. However, for such a $\tau$, we must have $\pi_{2}^{-1}(\tau) \cap \pi_{1}^{-1}(H) \cap \P_{\text{bad}} = \emptyset$, as it is a submanifold of $\pi_{1}^{-1}(H) \cap \P_{\text{bad}}$ of codimension $N$. It follows that $(H,\tau) \notin \P_{\text{bad}}$ for $H \in \G$ and $\tau \in \tau(H)$. By construction, $\G = \H^{l}_{\text{good}}(S)$. 

\subsubsection{Proof that $\rho_{S}$ is transverse to $\cZ$ \label{transversality construction}}

We will show that \(\rho_{S}\) is a submersion. Note that the Lie algebra of \(\Sp(T_{x_1}\Sigma)\) is \(\End_{0}(T_{x_1}\Sigma)\): the space of traceless linear endomorphisms of \(T_{x_1}(\Sigma)\). Hence, for \((H,\tau) \in  C^l([0,1]^N \times \R/\Z \times \Sigma ;P) \times [0,1]^N\),
\begin{equation*}
     T_{\rho_{S}(H,\tau)}\Sp(T_{x_1}\Sigma) = \{\rho_{S}(H,\tau) \circ M \mid M \in \End_0(T_{x_1}\Sigma)\}.
\end{equation*}
To prove that \(d\rho_{S}\) is a submersion at \((H,\tau)\), it suffices to find, for any \(M \in \End_{0}(T_{x_1}\Sigma)\), a path \(\gamma_{M}:\R \to C^l([0,1]^N \times \R/\Z \times \Sigma; P) \times [0,1]^N\) such that \(\gamma_M(0) = (H,\tau)\) and 
\begin{equation*}
    \frac{d}{dt} \Big|_{t = 0} \rho_S(\gamma_M(t)) = \rho_S(H,\tau) \circ M.
\end{equation*}
We may assume that \(\gamma_M(t)\) only varies the \(H\) factor and takes the form \(\gamma_M(t) = (H_{M}(t),\tau)\), where \({H_{M}(t)}^{\tau} = H^{\tau} \# tK\) for some \(K \in C^l(\R/\Z \times \Sigma; P)\). Furthermore, we assume \(K\) is supported in an open neighborhood around \(x_1\) that is disjoint from any other points in \(S\). A direct computation shows that
\begin{equation*}
   \rho_{S}(\gamma_M(t)) = \rho_S(H,\tau) \circ ({D\varphi_{K}^t})_{x_1} = \rho_S(H,\tau) \circ \exp(t \X_K) =  \rho_S(H,\tau) \circ \X_K
\end{equation*}
where
\begin{equation*}
    \X_K = \lim_{t \to 0}\frac{(D\varphi_K^{t})_{x_1} - \text{Id}}{t}
\end{equation*}
\( \) is the tangent vector at the identity of the one-parameter subgroup \(\{D\varphi_{K}^{t}\}_{t \in \R}\) in \(\Sp(T_{x_1}\Sigma, \omega)\). Hence, $\frac{d}{dt} \Big|_{t = 0} \rho_{S}(\gamma_M(t)) =  \rho_S(H,\tau) \circ \X_K$.  It remains to require that \(\X_K = M\). 

In local Darboux coordinates, assume that \(x_1 \sim (0,0)\) and that \(K(t, x,y) = ax^2 + bxy + cy^2\), where \(a, b, c \in \R\). Then, $\X_K = \begin{pmatrix}
    b & 2c \\
    -2a & -b
\end{pmatrix}$. Thus, the equation \(\X_K = M\) has a solution for any \(M \in \End_{0}(T_{x_1}\Sigma)\) (identified with $\End_{0}(\R^2)$ via the Darboux chart). 

\subsection{Proof of Lemma 5.2}
The proof follows the same construction as in the proof of Lemma \ref{PT inside P}. For any \(l \geq 3\), let \(C^l(\R/\Z \times \Sigma; P) = \{H \in C^{l}(\R/\Z \times \Sigma) \mid dH_{(t,p)} = 0 \text{ for all } t \in \R/\Z \text{ and } p \in P\}\). We want to show that there exists an open and dense subset of \(H \in C^{l}(\R/\Z \times \Sigma; P)\) such that \(S_1, \dots, S_k\) are non-degenerate periodic orbits of \(\Phi^H\). For \(i \in \{1, \dots, k\}\), let
\begin{equation*}
    \H_{\text{good}}(S_i) = \{H \in C^l(\R/\Z \times \Sigma; P) \mid S_i \text{ is a non-degenerate periodic orbit of } \Phi^H \}
\end{equation*}
Then \(\bigcap_{i=1}^{k} \H_{\text{good}}(S_i)\) is the desired set. To show it is open and dense, it suffices to show that each \(\H_{\text{good}}(S_i)\) is open and dense.

For each \(i \in \{1, \dots, k\}\), consider the map
\[
\rho_{S_i}: C^l(\R/\Z \times \Sigma; P) \to \Sp(T_{x_1}\Sigma, \omega)
\]
sending \(H\) to \(D(\Phi^H)^{d}_{x_1}\). Let \(\cZ \subset \Sp(T_{x_1}\Sigma, \omega)\) be the collection of linear symplectic automorphisms having 1 as an eigenvalue. Then \(\cZ = \cZ_{I} \cup \cZ_{II}\), where \(\cZ_{I} = \{\text{Id}\}\) and \(\cZ_{II}\) is a 2-dimensional submanifold of \(\Sp(T_{x_1}\Sigma, \omega)\). Note that $\H_{\text{good}}(S_i) = C^l(\R/\Z \times \Sigma; P) \setminus \rho_{S_i}^{-1}(\cZ)$. Since \(\cZ\) is closed, \(\H_{\text{good}}(S_i)\) is open. Furthermore, the same construction as in \ref{transversality construction} shows that \(\rho_{S_i}\) is a submersion. This implies the complement of \(\H_{\text{good}}(S_i)\) is a union of Banach submanifolds of \(C^l(\R/\Z \times \Sigma; P)\) of codimension at least 1. Thus, \(\H_{\text{good}}(S_i)\) is dense, as desired. 

\begin{section}{Proof of Theorem \ref{thm2} \label{Section 5}}
\subsection{}

Let \(C_{c}^{\infty}(S^2;P)\) be the space of smooth functions on \(S^2\) that are constant near \(P\). The following two propositions are adapted from \cite{Prasad}. Proposition \ref{NE discrete} is a corollary of Proposition \ref{NE ctns}.
\begin{prop}
\label{NE ctns}
Let $\Phi \in \Homeo(S^2,\omega)$ be such that $\Phi(P) = P$. Fix $\varepsilon > 0$ and a finite subset $\{F_i\}_{i=1}^{N}$ of $C^{\infty}(\R/\Z \times S^2 ;P)$. For any non-empty open set $U \subset C^{\infty}(\R/\Z \times S^2 ;P)$, there exist $H \in U$ and an orbit set $\O \in \P_{\R}(\Phi^H)$ such that 
\begin{equation*}
    \left \lvert \frac{1}{\lvert \mathcal{O} \rvert} \int_{\Theta_{\O}}F_i \,dt - \int_{0}^{1}\int_{S^2}F_i \,\omega \wedge dt \right \rvert < \varepsilon
    \end{equation*}
for any $i \in \{1,\dots, N\}$. Here, $\int_{\Theta_{\O}}F \,dt = \O(\tilde{F})$ where $\tilde{F}(\cdot) = \int_{0}^{1}F(t,\cdot) dt$.
\end{prop}

\begin{prop}
\label{NE discrete}
Let $\Phi \in \Homeo(S^2,\omega)$ be such that $\Phi(P) = P$. Fix $\varepsilon > 0$ and a finite subset $\{f_i\}_{i=1}^{N}$ of $C_{c}^{\infty}(S^2;P)$. For any nonempty open set $U \subset C^{\infty}(\R/\Z \times S^2 ; P)$, there exist $H \in U$ and an orbit set $\O \in \P_{\R}(\Phi^H)$ such that 
\begin{equation*}
    \left \lvert \frac{1}{\lvert \mathcal{O} \rvert} \O(f_i) - \int_{S^2} f_i \, \omega \right \rvert < \varepsilon
    \end{equation*}
for any $i \in \{1,\dots, N\}$. 
\end{prop}

\subsection{Proof of Proposition \ref{NE ctns}}
Let $\{\phi_n\}_{n \geq 1}$ be a sequence in $\Diff(S^2,\omega)$ such that $\phi_n(P) = P$ and  $\lim_{C^0}\phi_n = \Phi$. By taking $C^{\infty}$-small perturbations, we may assume that the $\phi_n$'s are non-degenerate. Choose \( H_n \in C^{\infty}(\R/\Z \times S^2) \) such that \( \phi_n = \varphi_{H_n}^{1} \).

\subsubsection{Multi-parameter Hamiltonians} As in \cite[Section 4.4.2]{Prasad}, we construct a multi-parameter Hamiltonian $F \in C^{\infty}([0,1]^N \times \R/\Z \times S^2 ;P)$ from $F_1,\dots,F_N$, such that $F^{\tau} \in U$ for each $\tau \in [0,1]^N$. By taking a $C^{\infty}$-small perturbation of $F$, we may assume that
\begin{equation*}
   \text{measure}(\{ \tau \in [0,1]^{N} \mid \text{\(\phi_{n}^{F^{\tau}}\) is non-degenerate for all $n$}\}) = 1
\end{equation*}
This follows from Lemma \ref{parametric transversality}.

\subsubsection{Application of Theorem \ref{double limit Weyl law}}

For each $(d,n) \in \N^{*} \times \N^{*}$, define the error function $e_{d, n}: C^{\infty}(\R/\Z \times S^2) \to \R$ by 
\begin{equation*}
    e_{d,n}(K) =  \int_{0}^{1}\int_{S^2}K \, \omega \wedge dt - \frac{c_d(H_n \# K) - c_d(H_n) + \int_{d \cdot \gamma_{H_n}}K \,dt}{d}
\end{equation*}
For $\tau \in [0,1]^{N}$, let
\begin{equation*}
    \overline{e}_{d,n}(\tau) = e_{d,n}(F^{\tau}) \text{ and } \overline{e}_{d}(\tau) = \limsup_{n \to \infty}\overline{e}_{d,n}(\tau)
\end{equation*}
By (\ref{the weyl law}), $\overline{e}_d \to 0$ pointwise on $[0,1]^{N}$. Moreover, by (\cite[Lemma 4.3]{Prasad}), for any $\tau, \tau' \in [0,1]^{N}$, we have
\begin{equation}
\label{lip condition}
    \lv \overline{e}_{d,n}(\tau) - \overline{e}_{d,n}(\tau') \rv \leq 2\lv \lv F^{\tau} -  F^{\tau'} \rv \rv_{C^0} \leq C \lv \lv \tau - \tau' \rv \rv
\end{equation}
where \( C > 0 \) is a Lipschitz constant that depends only on \( F \). Hence, \( | \overline{e}_{d}(\tau) - \overline{e}_{d}(\tau') | \leq C \| \tau - \tau' \| \). Since \( \{\overline{e}_d\}_{d \geq 1} \) is uniformly Lipschitz,  it converges to zero uniformly. 

\subsubsection{Uniform control of $e_{d,n}$}

Fix a sufficiently large \( d \) such that \( \| \overline{e}_{d} \|_{C^0} < \varepsilon \). Then, for each \( \tau \in [0,1]^{N} \), there is an index \( n_{\tau} \) such that \( \lv \overline{e}_{d,n}(\tau) \rv < \varepsilon/2 \) whenever \( n \geq n_{\tau} \). By (\ref{lip condition}), one can choose \( \delta > 0 \) such that for any \( (d,n) \in \N^{*} \times \N^{*}\), \( | \overline{e}_{d,n}(\tau) - \overline{e}_{d,n}(\tau') | < \varepsilon/2 \) whenever \( \| \tau - \tau' \| < \delta \). Select a finite subset \( \{\tau_1,\ldots,\tau_k\} \subset [0,1]^{N} \) such that the balls \( \{B_{\delta}(\tau_i)\} \) cover \( [0,1]^{N} \). Let \( n_0 = \max_{i \in \{1,\ldots, k\}} n_{\tau_i} \). Let \( \tau \in [0,1]^{N} \) and suppose \( \tau \in B_{\delta}(\tau_i) \). For any \( n \geq n_0 \), we have \( | \overline{e}_{d,n}(\tau) | < \varepsilon \) by the triangle inequality. Hence, \( \| \overline{e}_{d,n} \|_{C^0} < \varepsilon \) for all \( n \geq n_0 \). 

\subsubsection{A formal argument from \cite{Prasad}}

Apply the argument in \cite[Section 4.4.3]{Prasad} to all \( n \geq n_0 \) (the assumption on \( F \) in Section 5.2.1 is required). This produces a sequence \(\{\tau_n\}_{n \geq n_0}\) in \( [0,1]^{N} \) with corresponding orbit sets \( \mathcal{O}_n \in \mathcal{P}_{\mathbb{R}}(\phi_n^{F^{\tau_n}}) \), each with cardinality $d$, such that
\begin{equation}
\label{key equation}
\left \lvert \frac{1}{\lvert \mathcal{O}_n \rvert}\int_{\Theta_{\mathcal{O}_n}}F_i \, dt - \int_{0}^{1} \int_{S^2} F_i \, \omega \wedge dt \right \rvert  < \varepsilon
\end{equation}
for \( i \in \{1,\dots, N\} \). Furthermore, each \( \mathcal{O}_n \) is a convex combination of integral orbit sets. That is, \( \mathcal{O}_n = \sum_{j=1}^{N} a_{n,j} \mathcal{O}_{n,j} \), where the coefficients \( a_{n,j} \) are positive real numbers that satisfy \( \sum_{j=1}^{N} a_{n,j} = 1 \), and \( \mathcal{O}_{n,j} \in \mathcal{P}_{\mathbb{Z}}(\phi_{n}^{F^{\tau_n}}) \) with \( \lvert \mathcal{O}_{n,j} \rvert = d \).

\subsubsection{Convergence of orbit sets}
We show that a certain subsequence of $\{\O_n\}_{n \geq n_0}$ converges. By compactness, assume that $\tau_n \to \tau$. For each $j \in \{1,\dots,N\}$, write 
\begin{equation*}
    \O_{n,j} = \sum_{k=1}^{I(n,j)} b_{n,j,k} S_{n,j,k}
\end{equation*}
where $b_{n,j,k} \in \Z_{> 0}$ and $S_{n,j,k} \in \P(\phi_{n}^{F^{\tau_n}})$. Since \( \lvert \mathcal{O}_{n,j} \rvert \equiv d \), we may assume that \( \{I(n,j)\}_{n \geq 1}\) is a constant integral sequence taking value $I(j)$. Likewise, for each $k \in \{1,\dots, I(j)\}$, we assume that $\{b_{n,j,k}\}_{n \geq 1}$ and $\{\lvert S_{n,j,k} \rvert \}_{n \geq 1}$ are constant integral sequences taking values $b_{j,k}$ and $\lvert S_{j,k} \rvert$.

Since $\lim_{C^0}\phi_{n}^{F^{\tau_n}} = \phi^{F^{\tau}}$ and the orbits $\{S_{n,j,k}\}_{n \geq 1}$ have the same cardinality, it subsequentially converges to \( S_{j,k} \in \mathcal{P}(\Phi^{F^\tau})\). Finally, for each fixed $j$, a subsequence of \( a_{n,j} \) converges to \( a_{j} \), and the \( a_j \)'s satisfy \( \sum_{j=1}^{N} a_j = 1 \). Hence, a subsequential limit of $\{\O_n\}_{n \geq n_0}$ is
\begin{equation*}
\label{orbit set}
    \O  = \sum_{j=1}^{N}a_j\left(\sum_{k=1}^{I(j)} b_{j,k} S_{j,k}\right) \in \P_{\R}(\Phi^{F^\tau})
\end{equation*}
and $\lvert \O \rvert = d$. By (\ref{key equation}), we have
\begin{equation*}
    \left \lvert \frac{1}{\lvert \mathcal{O} \rvert} \int_{\Theta_{\O}}F_i \,dt - \int_{0}^{1}\int_{S^2}F_i \,\omega \wedge dt \right \rvert < \varepsilon
    \end{equation*}
for each $i \in \{1,\dots, N\}$. This finishes the proof.

\subsection{Proof of Theorem \ref{thm2}} 

Write \(S = S^2_{P}\). Since the restriction map \(C^{\infty}(S^2, S^2) \to C^{\infty}(S, S^2)\) is continuous with respect to the \(C^{\infty}_{\text{loc}}\)-topology, it suffices to prove the following claim.

\begin{claim}
\label{final claim}
Let $\phi \in \Diff(S,\omega_0)$. For $C^{\infty}$-generic $H \in C^{\infty}(\R/\Z \times S^2 ; P)$, $\phi \circ (\varphi_H^{1})|_{S}$ has an equidistributed sequence of orbit sets in $S$. 
\end{claim}

\subsubsection{Notation} 
Let $\Phi \in \Homeo(S^2,\omega)$ be the extension of $\phi$ with $\Phi(P) = P$.   If \(\O = \sum_{k}a_kS_k\) is an orbit set of \(\Phi\), define  
\begin{equation*}
    \mathcal{O}^{S} \coloneqq \sum_{k: S_k \subset S} a_k S_k \in \mathcal{P}_{\mathbb{R}}(\Phi|_{S}),
\end{equation*}
which is obtained by deleting the \( S_k \)'s that are subsets of \(P\).

\subsubsection{Set-up}
Let $\{f_i\}_{i \geq 1}$ be a countable, $C^0$-dense subset of $C_{0}^{\infty}(S^2;P)$. Let \(\{\varepsilon_N\}_{N \geq 1}\) be a sequence of positive real numbers with $\lim_{N \to \infty}\varepsilon_N = 0$. For each \(N \geq 1\), choose \(\varepsilon_{N}' \in (0,\frac{1}{4})\) such that
\begin{equation}
\label{epsilon assumption}
    \left (4\max_{i \in \{1,\dots,N\}}\lVert f_i \rVert_{C^0} + 1 \right) \cdot \frac{{\varepsilon_{N}}'}{1 - 4{\varepsilon_{N}}'} < \varepsilon_N
\end{equation}
In addition, choose $\chi_N \in C_{0}^{\infty}(S^2;P)$ such that $0 \leq \chi_N \leq 1$ and $\int_{S^2} (1 - \chi_N) \,\omega < {\varepsilon_{N}}'$. 

\subsubsection{}

For each $N \geq 1$, let $ \H(N)$ be the collection of $H \in C^{\infty}(\R/\Z \times S^2 ;P)$ such that there exists $\O \in \P_{\R}(\Phi^H)$ such that 
\begin{itemize}
        \item Every simple orbit in $\O^{S}$ is non-degenerate.
    \item For each $i \in \{1,\dots,N + 2\}$
\begin{equation}
\label{estimate}
\left \lvert \frac{1}{\lvert \mathcal{O} \rvert} \O(g_i) - \int_{S^2} g_i \, \omega \right \rvert \leq {\varepsilon_{N}}' 
\end{equation}
where $g_i = f_i$ for $i \in \{1,\dots,N\}$, $g_{N+1} = \chi_N$, $g_{N+2} = 1 - \chi_N$. 
\end{itemize}

Lemma \ref{non-degeneracy} and Proposition \ref{NE discrete} imply that $\H(N)$ is a dense in $C^{\infty}(\R/\Z \times S^2 ;P)$. Since each simple orbit in $\O^{S}$ is non-degenerate, $\H(N)$ is also open.  

\subsubsection{$\O^{S}$ is nearly equidistributed}
Fix \(N \geq 1\) and let \(H \in \H(N)\). Let $\O \in \P_{\R}(\Phi^H)$ such that $\O^{S}$ is non-degenerate and satisfies (\ref{estimate}).  By (\ref{estimate}) and the triangle inequality, we have 
\begin{align}
\label{rough estimate}
    \left \lvert \frac{1}{\lvert \O \rvert} \O^S(f_i) - \frac{\lvert \O^S \rvert}{\lvert \O \rvert}  \int_{S^2} f_i \,\omega  \right \rvert 
    &\leq  \left \lvert \frac{1}{\lv \O \rv}{\O(f_i)} - \int_{S^2} f_i \,\omega \right \rvert + \left \lvert \left(1 - \frac{\lvert \O^S \rvert}{\lvert \O \rvert} \right) \int_{S^2} f_i \, \omega \right \rvert \\
    &\leq {\varepsilon_{N}}' + \max_{i \in \{1,\dots,N\}}\lv \lv f_i \rv \rv_{C^0} \cdot \left \lvert \frac{\lvert \O^S \rvert}{\lvert \O \rvert} - 1 \right \rvert \nonumber
\end{align}
for $i \in \{1,\dots,N\}$. We now estimate the last term. Note that
\begin{equation*}
    \lvert \O^S \rvert = \O^S(1) = \O(\chi_N) + \O(1 - \chi_N)
\end{equation*}
Hence, by the triangle inequality, 
\begin{align}
\label{ratio estimate}
     \left \lvert \frac{\lvert \O^S \rvert}{\lvert \O \rvert} - 1 \right \rvert  & \leq \left \lvert \frac{\O^S(\chi_N)}{\lvert \O \rvert} - 1 \right \rvert + \frac{\O^S(1 - \chi_N)}{\lv \O \rv}  \nonumber \\ 
     &= \left \lvert \frac{\O^S(\chi_N)}{\lvert \O \rvert} - \int_{S^2}\chi_N \,\omega + \int_{S^2}(\chi_N - 1) \,\omega \right \rvert + \frac{\O^S(1 - \chi_N)}{\lv \O \rv} \nonumber \\
    & \leq \left \lvert \frac{\O(\chi_N)}{\lvert \O \rvert} - \int_{S^2} \chi_N \, \omega \right \rvert + \left \lvert \frac{\O(1 - \chi_N)}{\O} - \int_{S^2}(1- \chi_N) \, \omega \right \rvert + 2 \int_{S^2} (1 - \chi_N) \, \omega \nonumber \\ 
    & \leq 4{\varepsilon_{N}}' 
\end{align}
In the third step, we use that $0 \leq 1 - \chi_N \leq 1$. In the last step, we use (\ref{estimate}) for $g_{N+1} = \chi_N$ and $g_{N+2} = 1 - \chi_N$.  Since $\lv \O^{S} \rv \leq \lv \O \rv$, it follows from (\ref{ratio estimate}) that
\begin{equation}
    \frac{\lvert \O \rvert}{\lvert \O^S \rvert} \leq \frac{1}{1 - 4{\varepsilon_{N}}'} \label{ratio}
\end{equation}
By $(\ref{epsilon assumption}), (\ref{rough estimate})$, $(\ref{ratio estimate})$ and $(\ref{ratio})$, we have
\begin{align}
    \left \lvert \frac{1}{\lvert \O^S \rvert} \O^S(f_i) -  \int_{S^2} f_i \,\omega  \right \rvert  &=  \frac{\lvert \O \rvert}{\lvert \O^S \rvert} \cdot \left \lvert \frac{1}{\lvert \O \rvert} \O^S(f_i) - \frac{\lvert \O^S \rvert}{\lvert \O \rvert}  \int_{S^2} f_i \,\omega  \right \rvert \nonumber \\
    & \leq (4\max_{i \in \{1,\dots,N\}}\lv \lv  \overline{f_i} \rv \rv_{C^0} + 1) \cdot \frac{{\varepsilon_{N}}'}{1 - 4{\varepsilon_{N}}'}\\
    & < \varepsilon_N \nonumber
\end{align}
Note that \( \int_{S^2} f_i \, \omega = \int_{S} f_i \, \omega_0 \).
\end{section}

\subsubsection{Proof of Claim \ref{final claim}}
For $N \geq 1$, let $\H'(N)$ be the set of all $H \in C^{\infty}(\R/\Z \times S^2 ;P)$ such that there exists $\O^S \in \P_{\R}(\Phi^H|_{S})$ such that
\begin{equation}
\left \lvert \frac{1}{\lvert \O^S \rvert} \O^S(f_i) -  \int_{S} f_i \,\omega_0  \right \rvert < \varepsilon_N \label{estimate 2}
\end{equation}
for each $i \in \{1,\dots,N\}$. Let $\H_{good} = \bigcap_{N \geq 1}^{\infty}\H'(N)$. Since $\H(N) \subset \H'(N)$ and $\H(N)$ is open and dense, it follows that $\H_{good}$ is residual in $C^{\infty}(\R/\Z \times S^2 ;P)$. Fix $H \in \cH_{\text{good}}$. Let $\{\O_N\}_{N \geq 1}$ be a sequence of orbit sets of $\Phi^H|_{S}$ such that for each $N \geq 1$ and $i \in \{1,\dots,N\}$,
\begin{equation*}
\left \lvert \frac{1}{\lvert \O_N \rvert} \O_N(f_i) -  \int_{S} f_i \,\omega_0  \right \rvert < \varepsilon_N 
\end{equation*}
Let $f \in C_{0}^{\infty}(S^2;P)$. For each $N \geq 1$, we have
\begin{align*}
    \left \lvert \frac{1}{\lvert \O_N \rvert} \O_N(f) -  \int_{S} f \,\omega_0  \right \rvert \leq &  \inf_{i \in \{1,\dots,N\}} \Bigl( \left \lvert \frac{1}{\lvert \O_N \rvert} \O_N(f_i) -  \int_{S} f_i \,\omega_0  \right \rvert\\
    &+ \left \lvert \frac{1}{\lvert \O_N \rvert} \O_N(f - f_i) -  \int_{S} (f - f_i) \,\omega_0  \right \rvert \Bigr) \\
    & \leq \varepsilon_N + \inf_{i \in \{1,\dots,N\}}\lv \lv f - f_i \rv \rv_{C^0} \left( \left \lvert \frac{1}{\lvert \O_N \rvert} \O_N(1) \right \rvert + 1 \right) \\
    & \leq \varepsilon_N  + 2\inf_{i \in \{1,\dots,N\}}\lv \lv f - f_i \rv \rv_{C^0}
\end{align*}
by the triangle inequality. Taking the limit as \(N \to \infty\), we have
\begin{equation*}
    \lim_{N \to \infty} \left \lvert \frac{1}{\lvert \O_N \rvert} \O_N(f) -  \int_{S} f \,\omega_0  \right \rvert = 0
\end{equation*}
as desired. 

\begin{rem}
Ideally, we would like to use the same formal argument to prove a stronger statement: that a generic element of \(\text{Diff}(S, \omega_0)\) in the strong \(C^{\infty}\)-topology has an equidistributed sequence of orbit sets. However, the challenge arises because the restriction map \(C^{\infty}(S^2, S^2) \to C^{\infty}(S, S^2)\) is not continuous with respect to the strong \(C^{\infty}\)-topology. Consequently, for a \(C^{\infty}\)-small \(H \in C^{\infty}(\mathbb{R}/\mathbb{Z} \times S^2)\), the restriction \(\phi \circ (\varphi_H^{1})|_{S}\) might not be \(C^{\infty}\)-close to \(\phi\) in the strong topology.
\end{rem}

\nocite{*}
\bibliographystyle{alpha}

\end{document}